\documentclass[a4paper,superscriptaddress,11pt,accepted=9999-99-99,issuemonth=Janember, issueyear=9999, shorttitle=template]{compositionalityarticle}
\pdfoutput=1
\usepackage[english]{babel}
\usepackage[T1]{fontenc}
\usepackage{amsmath}
\usepackage{amsthm}
\usepackage{amsfonts}
\usepackage{hyperref}
\usepackage{tikz}
\usepackage{tkz-graph}
\usepackage{color}
\usepackage[square,numbers,sort&compress]{natbib}
\definecolor{purple}{RGB}{180,90,200}
\definecolor{dgreen}{RGB}{0,160,0}
\definecolor{turquoise}{RGB}{0,180,140}
\newcommand{\VV}[2]{\begingroup{\scriptscriptstyle%
    \left(\begin{smallmatrix}#1\\#2\end{smallmatrix}\right)}%
  \endgroup}
\newcommand{\AL}[3]{\mathfrak{#1}^{#2}_{#3}}

\newtheorem{theorem}{Theorem}
\newtheorem{lemma}[theorem]{Lemma}
\newtheorem{definition}[theorem]{Definition}

\newcommand{\Dtsub}{1} 
\newcommand{\Dtone}{2} 
\newcommand{\Dtcross}{4} 
\newcommand{\Dtrecover}{5} 

\newcommand{\Dpsub}{A1} 
\newcommand{\Dpone}{A2} 
\newcommand{\Dporder}{A3} 
\newcommand{\Dpcross}{A4} 
\newcommand{\Dprecover}{A5} 

\newcommand{\Dpcrosses}{\Dpcross$^*$} 

\begin{document}

\title{Alignments as Compositional Structures}
\date{}

\author{Sarah Berkemer}
\email{bsarah@bioinf.uni-leipzig.de}
\orcid{0000-0003-2028-7670}
\thanks{}
\affiliation{Max Planck Institute for Mathematics in the Sciences, Leipzig,
  Germany; and Bioinformatics Group, Department of Computer Science,
  Universit{\"a}t Leipzig, Germany}

\author{Christian H{\"o}ner zu Siederdissen}
\email{choener@bioinf.uni-leipzig.de}
\orcid{0000-0001-9517-5839}
\thanks{}
\affiliation{Bioinformatics Group, Department of Computer Science,
    Universit{\"a}t Leipzig, Germany}

\author{Peter F.\ Stadler}
\email{studla@bioinf.uni-leipzig.de}
\orcid{0000-0002-5016-5191}
\thanks{}
\affiliation{Bioinformatics Group, Department of Computer Science,
  Universit{\"a}t Leipzig, Germany; Max Planck Institute for Mathematics
  in the Sciences, Leipzig, Germany; Department of Theoretical Chemistry,
  University of Vienna, Austria; Facultad de Ciencias, Universidad Nacional
  de Colombia, Bogot{\'a}, Colombia; Santa Fe Institute, Santa Fe, NM, USA;
}
\maketitle

\begin{abstract}
  Alignments, i.e., position-wise comparisons of two or more strings or
  ordered lists are of utmost practical importance in computational biology
  and a host of other fields, including historical linguistics and emerging
  areas of research in the Digital Humanities. The problem is well-known to
  be computationally hard as soon as the number of input strings is not
  bounded. Due to its practical importance, a huge number of heuristics
  have been devised, which have proved very successful in a wide range of
  applications. Alignments nevertheless have received hardly any attention
  as formal, mathematical structures. Here, we focus on the compositional
  aspects of alignments, which underlie most algorithmic approaches to
  computing alignments. We also show that the concepts naturally generalize
  to finite partially ordered sets and partial maps between them that in
  some sense preserve the partial orders.
\end{abstract}

\section{Introduction}

Alignments play an important role in particular in bioinformatics as a
means of comparing two or more strings by explicitly identifying
correspondences between letters (usually called matches and mismatches) as
well as insertions and deletions \cite{Durbin:98}. The aligned positions
are interpreted either as deriving from a common ancestor (``homologous'')
or to be functionally equivalent. Alignments have also been explored as
means of comparing words in natural languages, see e.g.\
\cite{Kondrak:00,Cysouw:07,Steiner:11a,Bhattacharya:18a}, as a convenient
way of comparing ranked lists \cite{Fagin:03}, for comparison of text
editions \cite{Wolff:00,Tiepmar:17}, and to analyse synteny in the
comparison of genomes \cite{Grabherr:10,VelandiaHuerto:16b}.

The literature on alignments is extensive. However, it its concerned almost
exclusively with practical algorithms and applications. The alignment
problem for two input strings has an elegant recursive solution for rather
general cost models and has served as one of the early paradigmatic
examples of dynamic programming \cite{Needleman:70,Sankoff:83}. Since these
algorithms have only quadratic space and time requirements for simple cost
models \cite{Needleman:70,Gotoh:82}, they are of key importance in
practical applications. The same recursive structure easily generalizes to
alignments of more than two sequences \cite{Carillo:88,Lipman:89} even
though the cost models need to be more restrictive to guarantee
polynomial-time algorithms \cite{Kececioglu:04}. The computational effort
for these exact solutions to the alignment problem increases exponentially
with the number of sequences, hence only implementations for 3-way
\cite{Gotoh:86,Konagurthu:04,Kruspe:07a} and 4-way alignments
\cite{Steiner:11a} have gained practical importance. A wide variety of
multiple sequence alignment problems (for arbitrary numbers of input
sequences) have been shown to be NP-hard
\cite{Kececioglu:93,Wang:94,Bonizzoni:01,Just:01,Elias:06} and MAX SNP-hard
\cite{Wareham:95,Manthey:03}. The construction of practical multiple
alignment algorithms therefore relies on
heuristic approximations. These fall into several classes, see e.g.\
\cite{Edgar:06,Baichoo:17} for reviews.
  
\par\noindent(1) \emph{Progressive} methods typically compute all pairwise
alignments and then use a ``guide tree'' to determine the order in which
these are stepwisely combined into a multiple alignment of all input
sequences. The classical example is \texttt{ClustalW} \cite{Larkin:07}.
The approach can be extended to starting from exact 3-way
\cite{Konagurthu:04,Kruspe:07a} or 4-way alignments \cite{Steiner:11a}.
\par\noindent(2) \emph{Iterative} methods starting to align small gapless
subsequences and then extend and improve the alignment until the score
converges. A paradigmatic example is \texttt{DIALIGN}
\cite{Morgenstern:99b}.
\par\noindent(3) \emph{Consistency}-based alignments and \emph{consensus}
methods start from a collection of partial alignments (often exact pairwise
alignments) to obtain candidate matches and extract a multiple alignment
using agreements between between the input alignments.

Most of the successful multiple alignment algorithm in computational
biology combine these paradigms. For example \texttt{T-COFFEE}
\cite{Notredame:00} and \texttt{ProbCons} \cite{Do:05} use consistency
ideas in combination with progressive constructions; \texttt{MUSCLE}
\cite{Edgar:04} and \texttt{MAFFT} \cite{Katoh:05} combine progressive
alignments with iterative refinements.

\begin{figure}
  \begin{minipage}{0.9\textwidth}
\begin{verbatim} 
        (a)               (b)               (c)                (d)
  A 0000111110000    A 0000111110000    B 000011011----    B 000011011
  B 000011011----    C ----100010000    C ----100010000    C 100010000
    s = 4              s = 2              s = -5             s = 5
\end{verbatim}
  \end{minipage}
  \caption{Alignments of three binary sequences \texttt{A}, \texttt{B}, and
    \texttt{C} with a simple column-wise score of $+1$ for matches, $0$ for
    mismatches, and $-1$ for gaps. Alignment (c) is transitively implied by
    (a) and (b), but is it not an optimal pairwise alignment of \texttt{B}
    and \texttt{C}.}
\label{fig:cxample0}
\end{figure}
  
A key assumption underlying consistency based methods is transitivity:
considering three input sequences $x$, $y$, and $z$, if $x_i$ aligns with
$y_j$ and $y_j$ aligns with $z_k$, then $x_i$ should also align with
$z_k$. While this property holds for the pairwise constituents of a
multiple alignment, it is a well known fact that the three score-optimal
alignments that can be constructed from three sequences in general violate
transitivity, see Fig.~\ref{fig:cxample0}. \texttt{TRANSALIGN}
\cite{Malde:13} uses transitivity to align input sequences to a target
database using an intermediary database of sequences to increase the search
space. Here, intermediary sequences show which subsequences of input and
target sequence can be transitively aligned. This may result in a few well
aligned subsequences that are then extended to one aligned region via a
simple scoring function. The same notion of transitivity is also used in
\texttt{psiblast} \cite{Altschul:97} to stepwisely increase the set of
sequences that are faintly similar to an input sequence.

Practical applications distinguish whether the complete input sequences are
to be aligned, or whether a maximally scoring interval is to be
considered. In the latter case one allows an additional ``unaligned state''
for prefixes and/or suffixes of the input. This leads to slight changes in
exact algorithms, exemplified by an extra term in the local Smith-Waterman
algorithm \cite{Smith:81} compared to the global Needleman-Wunsch
\cite{Needleman:70} algorithm. This idea can be generalized to mixed
problems in which a user can determine for each of the two ends of each
input sequence whether it is to be treated as local or global
\cite{Retzlaff:18a}. For the purpose of the present contribution
(partially) local alignments require a slight, trivial extension of the
presentation, which we -- for the sake of clarity in the presentation --
only briefly comment on.

Alignments are usually constructed from strings or other totally ordered
inputs, hence the columns of the resulting alignment are usually also
treated as a totally ordered set. Consecutive insertions and deletions,
however, are not naturally ordered relative to each other:
\begin{equation}
  \begin{minipage}{0.7\textwidth} 
\begin{verbatim}
  gugugu--acgggcca          guguguac--gggcca 
  gucuguug--gggccc          gucugu--uggggccc
\end{verbatim} 
  \end{minipage}
  \label{eq:alncolumns}
\end{equation}
are alignments that are equivalent under most plausible scoring models.
The idea to consider alignment columns as partial orders was explored
systematically in \cite{Lee:02} and a series of follow-up publications
\cite{Lee:03,Grasso:04}. Here, (mis)matches are considered as an ordered
backbone, with no direct ordering constraints between an insertion and a
deletion. The resulting alignments are then represented as directed acyclic
graphs (DAGs), more precisely, as the Hasse diagrams of the partial
order. The key idea behind the \texttt{POA} software \cite{Lee:02} is that
a sequence of DAGs can be used as an input to a modified version of the
Needleman-Wunsch algorithm \cite{Needleman:70}. Recently this idea has been
generalized to the problem of aligning a sequence to a general directed
graph \cite{Rautiainen:17,Vaddadi:17}.

Despite the immense practical importance of alignments, they have received
very little attention as mathematical structures in the past. The most
comprehensive treatment, at least to our knowledge, is the Technical Report
\cite{Morgenstern:99}, which considers (pairwise) alignments as binary
relations between sequence positions that represent matchings and preserve
order. We use many of these ideas here.  The notion of a composition of
pairwise alignments -- formalized as composition of partial maps that
represent the matching -- first appears in \cite{Malde:13}. We will return
to this point in Section~\ref{sect:composition}.  Following our earlier
work \cite{Otto:11a}, we will use a language that is closer to graph theory
than the presentation of \cite{Morgenstern:99}.

\section{Alignments and Partial Orders}

Consider a finite collection $X$ of two or more finite totally ordered sets
$X_a$. It will be convenient in the following to denote an element
$i\in X_a$ by $(a,i)$. The following definition rephrases the approach
taken e.g.\ in \cite{Stoye:97,Morgenstern:99}. It will be generalized below
to deal with partial orders instead of total orders.
\begin{definition}[\cite{Otto:11a}]
  A total alignment of the totally ordered sets $X_a$ is a triple $(X,A,<)$
  where $(X,A)$ is a graph and $<$ is a total order relation on the set of
  connected components $\mathcal{C}(X,A)$ satisfying\footnote{There is no
    condition (3) due to synchronization with the definitions for partial
    orders defined later.}
  \begin{itemize}
  \item[{\rm(\Dtsub)}] $Q\in\mathcal{C}(X,A)$ is a complete subgraph of
    $(X,A)$.
  \item[{\rm(\Dtone)}] If $(a,i)\in Q$ and $(a,j)\in Q$, then $i=j$.
  \item[{\rm(\Dtcross)}] If $(a,i),(b,j)\in P$ and $(a,k),(b,l)\in Q$
    with $i<k$ then $j<l$.
  \item[{\rm(\Dtrecover)}] If $(a,i)\in P$, $(a,j)\in Q$ and $i<j$ then
    $P<Q$.
  \end{itemize}
  \label{def:totA1}
\end{definition}
The connected components of $(X,A)$ are usually called the alignment
columns. Condition (\Dtone) ensures that every alignment column contains at
most one element of each ordered set $X_a$.  Conversely, every element
$(a,i)$ is contained in exactly one connected component, i.e., alignment
column.  Condition (\Dtcross) requires that alignment columns do not cross.
Condition (\Dtrecover) ensures that the order on the columns is such that
the projection of the alignment columns to each individual row exactly
recovers the input order.  Conditions (\Dtcross) and (\Dtrecover) in
general only specify a partial order as the following result shows:
\begin{lemma}
  \label{lem:totpo}
  Let $(X,A)$ be the graph of an alignment and denote by $\prec$ the
  relation on $\mathcal{C}(X,A)$ defined by $P\prec Q$ whenever there is
  $(a,i)\in P$ and $(a,j)\in Q$ with $i<j$. Then the transitive closure
  $\ddot\prec$ of $\prec$ is a partial order on $\mathcal{C}(X,A)$.
\end{lemma}
\begin{proof}
  Clearly $\ddot\prec$ is antisymmetric. If $P\prec Q$, then there there is
  a sequence of columns $P=Q_0\ddot\prec Q_1\ddot\prec \dots Q_k=Q$. Since
  the sequence of elements $(a,i)$ belonging to the same $X_a$ is strictly
  increasing with the column index $j$ for each $a$ along any such sequence
  of columns, it follows that the transitive closure of $\ddot\prec$ is
  still antisymmetric, and thus a partial order.
\end{proof}
As an immediate consequence, there is also a (not necessarily unique)
total order $<_*$ of the alignment columns, obtained as an arbitrary linear
extension of $\ddot\prec$, which by construction satisfies
\begin{equation}
  \label{eq:total}
  P <_* Q,\,
  (a,i)\in P,\,\textrm{and}\, (a,j)\in Q \quad\textrm{implies}\quad i<j. 
\end{equation}
Hence, whenever conditions (\Dtsub), (\Dtone), and (\Dtcross) in
Definition~\ref{def:totA1} are satisfied, there indeed exists a total order
on $\mathcal{C}(X,A)$ that satisfies condition (\Dtrecover).

In order to treat (partially) local alignments it is necessary to
distinguish aligned and ``unaligned'' columns. Each unaligned column may
contain only a single element -- note however, that also regular columns
may contain only a single entry from each row.  Furthermore, all
``unaligned'' positions for a prefix and/or a suffix of each input
$(X_a,<_a)$ form ``unaligned'' columns.

In this condition we will consider a more general setting. Instead of
totally ordered sets $X_a$ we will consider finite partially ordered sets
$(X_a,\prec_a)$.
\begin{definition}
  An alignment of $X$ is a triple $(X,A,\prec)$ where $(X,A)$ is a graph
  and $\prec$ is a partial order on the set of connected components
  $\mathcal{C}(X,A)$ such that 
  \begin{itemize}
  \item[{\rm(\Dpsub)}] $Q\in\mathcal{C}(X,A)$ is a complete subgraph of
    $(X,A)$.
  \item[{\rm(\Dpone)}] If $(a,i)\in Q$ and $(a,j)\in Q$, then $i=j$.
  \item[{\rm(\Dporder)}] If $(a,i)\in P$, $(a,j)\in Q$ and
    $(a,i)\prec_a (a,j)$ then $P\prec Q$.
  \item[{\rm(\Dpcross)}] $P\prec Q$, $(a,i)\in P$ and $(a,j)\in Q$ implies
    $(a,i)\prec_a (a,j)$ or $(a,i)$ and $(a,j)$ are incomparable w.r.t.\
    $\prec_a$.
  \end{itemize}
  \label{def:A}
\end{definition}
Condition (\Dporder) constrains the partial order on the columns to respect
the partial order of the rows. Condition (\Dpcross) insists that columns
also must not cross indirectly.

Condition (\Dpcross) obviously implies the following generalization of
(\Dtcross):
\begin{itemize}
\item[(\Dpcrosses)] $(a,i), (b,j)\in P$ and $(a,k),(b,l)\in Q$ and
  $(a,i)\prec_a (a,k)$ implies $(b,j)\prec_b(b,l)$ or $(b,j)$ and $(b,l)$ are
  incomparable w.r.t.\ $\prec_b$. 
\end{itemize}
However, (\Dpcrosses) is not sufficient to guarantee that the alignment columns
form a partially ordered set. A counterexample is shown in
Fig.~\ref{fig:counterx1}. It is therefore necessary to require the
existence of the partial order $\prec$ on $\mathcal{C}(X,A)$ in
Definition~\ref{def:A}.

\begin{figure}
  \begin{center}
    \includegraphics[width=0.7\textwidth]{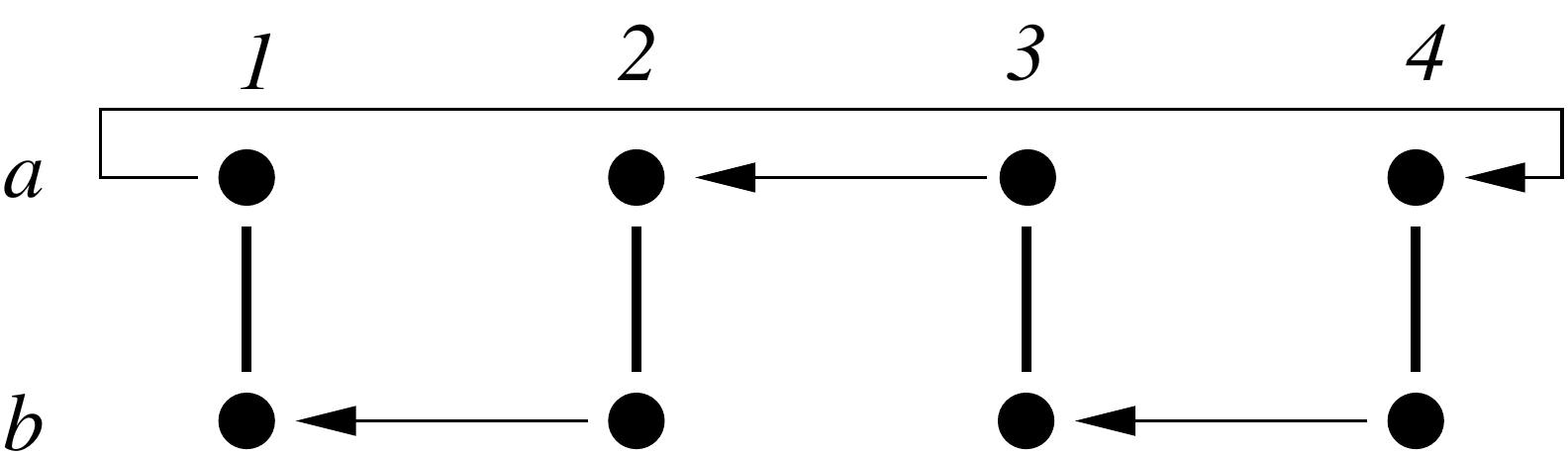}
  \end{center}
  \caption{Property (\Dpcrosses) is not sufficient to ensure the existence
    of a partial order $\prec$ on $\mathcal{C}(X,A)$. Consider the partial
    orders $(a,4)\prec_a (a,1)$ and $(a,2)\prec_a (a,3)$ and
    $(b,1)\prec_b (b,2)$ and $(b,3)\prec_a (b,4)$, with alignment colums
    $\{(a,i),(b,i)\}$ for $i=1,2,3,4$. Clearly (\Dpone), (\Dporder), and
    (\Dpcrosses) holds, but the directed cycle shows that no partial order
    on the colums exists that is consistent with both partial orders.  }
  \label{fig:counterx1}
\end{figure}

In order to model partially local alignments of posets we consider the set
$\mathcal{A}$ of aligned columns and a partition of the set of ``unaligned
columns'' into two not necessarily non-empty subsets $\mathcal{P}$ and
$\mathcal{S}$ such that for all $U\in\mathcal{P}$, $V\in\mathcal{A}$ and
$W\in\mathcal{S}$ it holds that $W\not\preceq V$ and $V\not\preceq U$,
i.e., no ``unaligned'' suffix column preceeds an aligned column, and no
``unaligned'' prefix column succeeds an aligned column. ``Unaligned''
prefix columns belonging to different rows $(X_a,\prec_a)$ are considered
mutually incomparable; the same is assumed for ``unaligned'' suffix
columns. With the caveat that ``unaligned'' columns need to be marked as
such, there is no structural difference between local and global
alignments.

If all $(X_a,\prec_a)$ are totally ordered then condition (\Dpcross)
implies the non-crossing condition (\Dtcross) because $(b,j)$ and $(b,l)$
cannot be incomparable w.r.t.\ $\prec_b$, and thus the required partial
order $\prec$ is obtained as the transitive closure of the relative order
of any two columns. Definitions \ref{def:totA1} and \ref{def:A} therefore
coincide for totally ordered rows.

The existence of (non-trivial) alignments of any collection of finite
partial orders ${(X_i,\prec_i)}$, $i=1,\dots,N$ is easy to see: each of the
partial orders can be linearly extended to a total order $(X_i,<_i)$. Any
alignment of these total orders is also an alignment of the underlying
partial orders, with a suitable partial order of the columns given by
Lemma~\ref{lem:totpo}.

It may be interesting to explore alignments satisfying a (much) stronger
version of axiom (\Dpcross), which stipulates that $(X_a,\prec_a)$ is
recovered as projection of $(X,A)$ onto row $a$, i.e.,
\begin{itemize}
\item[(\Dprecover)] $P\prec Q$, $(a,i)\in P$ and $(a,j)\in Q$ implies
  $(a,i)\prec_a(a,j)$.
\end{itemize} 
As argued above, (\Dpcross) and (\Dprecover) are equivalent if all
$(X_a,\prec_a)$ are totally ordered. In general this is not the case, as
the example in Fig.~\ref{fig:counter2} shows.

\begin{figure}
  \begin{center}
    \includegraphics[width=0.95\textwidth]{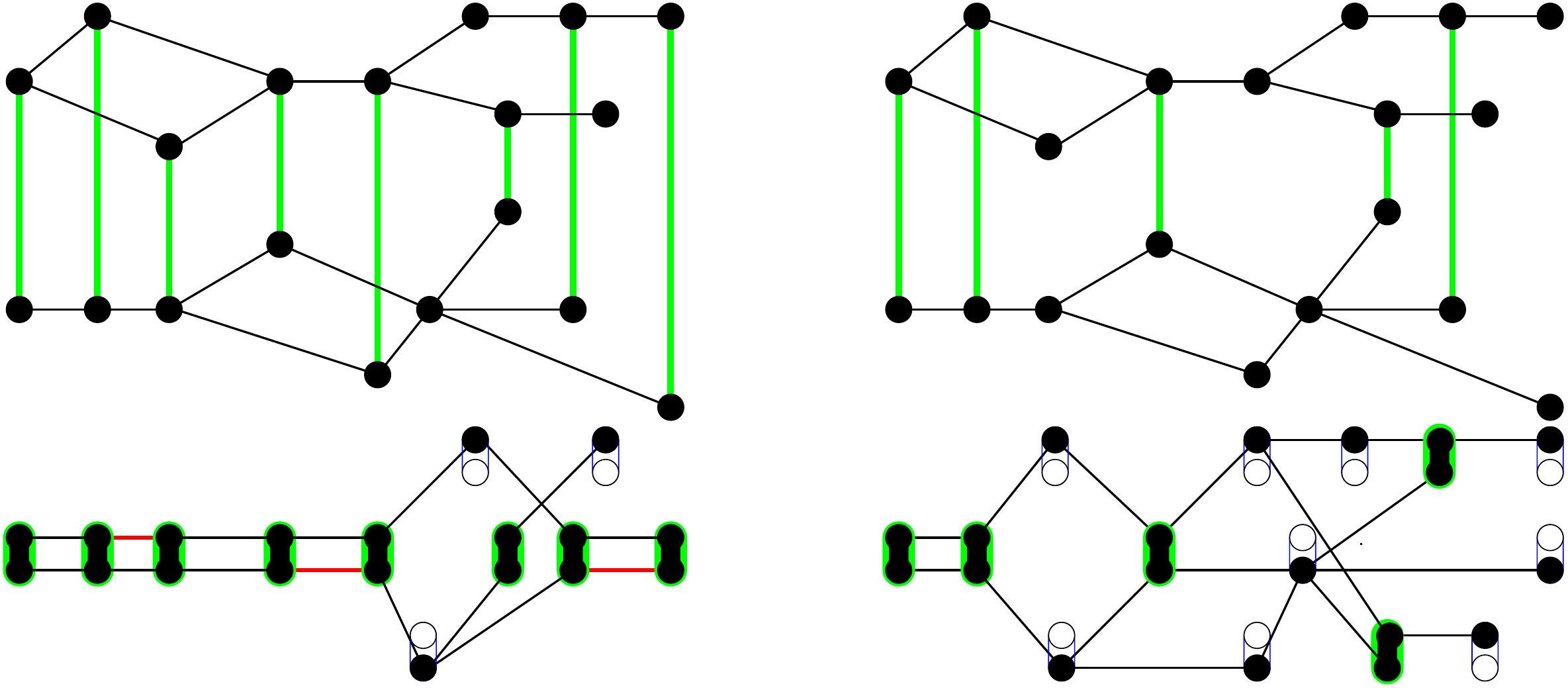}
  \end{center}
  \caption{\textbf{Top:} Pairwise alignments of partially ordered
    sets. Thin black edges show the Hasse diagram, to be read from left to
    right. Alignment edges are shown in green.\newline%
    \textbf{Bottom:} The induced partial order of the alignment columns
    with corresponding points vertically aligned. The partial order is
    again shown as a Hasse diagram, with superflous edges omitted. Both the
    l.h.s.\ and the r.h.s.\ example satisfy (\Dpcross), i.e., none of order
    relations $\prec_1$ and $\prec_2$ is violated in the alignment. The red
    edges highlight two comparabilities introduced by partial order of the
    columns that are absent in the input posets.  Red edges therefore imply
    a violation of condition (\Dprecover). Hence the l.h.s.\ alignment
    violates (\Dprecover), while the r.h.s.\ alignment does not.}
  \label{fig:counter2} 
\end{figure}

The following simple, technical result is a generalization of
Lemma~\ref{lem:totpo}, showing that condition (\Dprecover) is sufficient to
guarantee the existence of a partial order on the columns.
\begin{lemma}
  \label{lem:po2}
  Let $(X,A)$ be a graph with connected components $\mathcal{C}(X,A)$
  satisfying {\rm(\Dpsub)} and {\rm(\Dpone)}. Let $\prec$ denote the
  transitive closure of the relation $\dot\prec$ defined by
  {\rm(\Dporder)}, i.e., $P\dot\prec Q$ whenever $(a,i)\in P$, $(a,j)\in Q$
  and $(a,i)\prec_a (a,j)$ then $P\prec Q$.  Finally assume that axiom
  {\rm(\Dprecover)} holds. Then $\prec$ is a partial order on
  $\mathcal{C}(X,A)$
\end{lemma}
\begin{proof}
  It suffices to show that $\prec$ is antisymmetric. It is clear from the
  construction that by (\Dprecover) we know that $\dot\prec$ is
  antisymmetric.  If $\prec$ is not antisymmetric, then there is a finite
  sequence of columns $P_i$, $i=0,\dots, k$ such that
  $P_0\dot\prec P_1\dot\prec\dots\dot\prec P_k\dot\prec P_0$ such that any
  two consecutive columns $P_i$ and $P_{i+1}$ have at a pair of entries,
  say $(a_i,h)\in P_i$ and $(a_i,h')\in P_{i+1}$, in the same row. For the
  transitive closure this would imply both $(a_i,h)\prec(a_i,h')$ from
  $(a_i,h)\dot\prec(a_i,h')$ and $(a_i,h')\prec(a_i,h)$ by going around the
  cycle, contradictiong axiom (\Dprecover).
\end{proof}

Condition (\Dprecover) implies that the restriction of the partial order
$\prec$ on the columns to any subset of columns in which a given set of
rows is represented coincides with the induced partial order on the
corresponding vertex set in $(X_a,\prec_a)$. Regarding the $(X_a,\prec_a)$
as graphs, the aligned columns form a common induced subgraph. The
alignment problem for partially ordered sets under axiom (\Dprecover) thus
can be seen as a generalized version of a maximum induced subgraph
problem. We refer to \cite{Bunke:97} for a discussion of the relationships
of edit distances and maximum common subgraph problems in a more general
setting.

The following result generalizes Lemma~1 of \cite{Otto:11a}:
\begin{lemma}
  \label{lem:subset}
  Let $(X,A,\prec)$ be an alignment and let $Y\subseteq X$. Then the
  induced subgraph $(X,A)[Y]$ with the partial order $\prec$ restricted to
  the non-empty intersections $Q\cap Y$ for $Q\in\mathcal{C}(X,A)$ is again
  an alignment. Furthermore, if $(X,A,\prec)$ satisfies {\rm(\Dprecover)},
  then the restriction to $(X,A)[Y]$ again satisfies {\rm(\Dprecover)}.
\end{lemma}
\begin{proof}
  Every induced subgraph of a complete graph is again a complete graph,
  hence (\Dpsub) holds for $(X,A)[Y]$, hence the connected components of
  $(X,A)[Y]$ are exactly the non-empty intersections of $Y$ with the
  components $Q$ of $(X,A)$.  Condition (\Dpone) remains unchanged by the
  restriction to $Y$. Finally, the partial order $\prec$ satisfying
  (\Dporder) restricted to the non-empty intersections $Q\cap Y$ for
  $Q\in\mathcal{C}(X,A)$ is a partial order that obviously still satisfies
  (\Dpcross) since the restriction to $Y$ only removes some of the
  conditions in (\Dpcross).

  To see that the restriction of $(X,A)[Y]$ again satisfies (\Dprecover) it
  suffices to recall that the partial order in the colums is given by
  $P\cap Y \prec Q\cap Y$ whenever $P\prec Q$ and both
  $P\cap Y\ne\emptyset$ and $Q\cap Y\ne\emptyset$. If one of the
  intersections is empty, axiom (\Dprecover) becomes void since the empty
  set is not a column in $(X,A)[Y]$. On the other hand, if the two
  restricted columns have entries $(a,i)$ and $(a,j)$ in the same row, then
  (\Dprecover) for $(X,A,\prec)$ ensures $(a,i)\prec_a(a,j)$, i.e., the
  implication (\Dprecover) remains true for the restricted alignment.
\end{proof}

Note that additional partial orders on connected components of the induced
subgraph $(X,A)[Y]$ may exist that are not obtained as restrictions of the
partial order on $\mathcal{C}(X,A)$. The reason is that omitting parts of
the columns may allow a relaxation of their mutual ordering. 

Rooted trees can be seen as partially ordered sets, with the natural
partial order defined by $x\prec y$ if $y$ lies on the unique path
connecting $x$ and the root of the tree. This special case is thus covered
in the general framework outlined here. Usually, tree alignments are
defined on rooted \emph{oriented trees}, however, where the relative order
of siblings is preserved \cite{Jiang:95,Hoechsmann:04,Berkemer:17}, thus
imposing additional restrictions on valid alignments. We will return to
this point in some generality in the discussion section.

\section{Composition of Alignments} 
\label{sect:composition}

The fact that alignments are again totally or partially ordered sets
implies that one can also meaningfully define alignments of alignments.
More precisely:

\begin{lemma}
  Let $(X,A,\prec)$ be an alignment and consider a non-trivial partition
  $\mathfrak{P}$ of the set of objects, i.e., the rows. Denote the site
  sets of the classes of $\mathfrak{P}$ by $X_1$, $X_2$, \dots, $X_p$ and
  consider the sub-alignments $(X,A,\prec)[X_i]$.  Then $(X,A,\prec)$ is
  isomorphic to the (vertex) disjoint union of the $(X,A,\prec)[X_i]$
  augmented by extra edges $(x',x'')$ whenever there is a column $Q$ of $A$
  with $x'\in Q\cap X_i$ and $x''\in Q\cap X_j$ for $X_i\ne X_j$.
\end{lemma}
\begin{proof}
  The alignments $(X,A,\prec)[X_i]$ are induced subgraphs of $(X,A)$. Their
  disjoint union lacks exactly all edges that connect pairs of vertices
  that are in the same connected component of $(X,A)$ but are not in the
  same subgraph $(X,A)[X_i]$. Since the partial order on the colums of
    $(X,A)[X_i]$ is the one inherted from $(X,A,\prec)$, the re-composition
    of the columns also recovers the original partial order.
\end{proof}

The $(X,A,\prec)[X_i]$ can also be interpreted as partially ordered sets
whose \emph{points} are the non-empty restrictions $Q\cap X_i$ of the
connected components of $(X,A)$.
\begin{definition}
  We denote by $(X,A)/\mathfrak{P}$ the quotient graph whose vertices are
  the columns of the alignments $(X,A)[X_i]$, that is, the non-empty sets
  $Q\cap X_i$ where $Q$ is a connected component of $(X,A)$. Its edges are
  the pairs $(Q\cap X_i,Q\cap X_j)$ for which both $Q\cap X_i$ and
  $Q\cap X_j$ are non-empty. 
\end{definition}
The connected components of the graph $(X,A)/\mathfrak{P}$ are therefore of
the form $Q' := Q/\mathfrak{P} = \{ Q\cap X_i | Q\cap
X_i\ne\emptyset\}$. Note that $Q'$ is non-empty since the column $Q$ of
$(X,A)$ contains at least one element, which belongs to at least one of the
$(X,A)[X_i]$. Thus there is a 1-1 correspondence between the connected
components of $(X,A)$ and those of $(X,A)/\mathfrak{P}$. The columns of
$(X,A)/\mathfrak{P}$ naturally inherit the partial order $\prec$ of
$\mathcal{C}(X,A)$. We write $(X,A,\prec)/\mathfrak{P}$ for the quotient
graph with this partial order on its connected components.
\begin{lemma}
  \label{lem:quotient}
  $(X,A,\prec)/\mathfrak{P}$ is an alignment.
\end{lemma}
\begin{proof}
  Consider the quotient graph $(X,A)/\mathfrak{P}$. By construction, each
  column $Q'$ is a complete graph and contains at most one node for each
  class of $\mathfrak{P}$ since it is the quotient of a column of
  $(X,A,\prec)$ w.r.t.\ $\mathfrak{P}$. Also by construction, we have
  $P'\prec Q'$ for the columns of $(X,A)/\mathfrak{P}$ whenever $P \prec Q$
  in $(X,A,\prec)$. Since there is a 1-1 correpondence between columns of
  $(X,A,\prec)$ and $(X,A,\prec)/\mathfrak{P}$, $\prec$ also serves as a
  partial order on the columns of $(X,A)/\mathfrak{P}$, which is by
  construction consistent with the partial order on each of the
  $(X,A)[X_i]$.
\end{proof}

As a consequence, every alignment can be decomposed into an alignment of
alignments w.r.t.\ an arbitrary partition of the rows. The constituent
alignments $(X_i,A,\prec_i)$ have at most the same number of columns since
``all gap'' columns, $Q'=Q\cap X_i=\emptyset$, are removed. By
Lemma~\ref{lem:quotient}, the decomposition can be used recursively until
each constituent is only a single partially ordered set $(X_a,\prec_a)$.
Any such recursive composition is naturally represented as a tree
$\mathfrak{T}$ whose leaves are the input posets $(X_a,\prec_a)$.  Each
internal node of $\mathfrak{T}$ corresponds the an alignment of its
children, with the root corresponding to $(X,A,\prec)$, the alignment of
all the data.

The reverse of this type of decomposition underlies all \emph{progressive
  alignment} schemes.  One starts from a guide tree $\mathfrak{T}$ whose
leaves are the $(X_a,\prec_a)$ and for each inner node of $\mathfrak{T}$
constructs an alignment (or a set of alternative alignments) from the (set
of) alignments attached to its children.  It is important to note that a
score-optimal alignment $(X,A,\prec)$ in general is \textbf{not} the
score-optimal alignment $(X,A,\prec)/\mathfrak{P}$ of score-optimal
consitutents $(X_i,A_i,\prec_i)$, or, in other words, if $(X,A,\prec)$ is
score-optimal, there is no guarantee that there is any partition of the
rows $\mathfrak{P}$ such that all the restrictions $(X,A,\prec)[X_i]$ are
score-optimal subalignments. Progressive alignments methods thus can only
approximate the solution of the multiple alignment problem. Practical
results depend substantially on the choice of the guide tree
$\mathfrak{T}$. It is has been suggested early \cite{Feng:87}, that
$\mathfrak{T}$ should closely resemble the evolutionary history of the
input sequences. Usually $\mathfrak{T}$ is constructed from distance or
similarity measures between all pairs of input sequences -- and usually
pairwise alignments are employed to obtain these data. A special case of
progressive alignment adds a single sequence in each step, instead of also
considering alignments of alignments.

\section{Blockwise Decompositions}

On the other hand, we can also decompose alignments into blocks of columns.
More precisely, if $(X,A,\prec)$ is an alignment and $\mathfrak{Q}$ is a
partition of the $X$ with classes $Y_k$ such that
\begin{itemize}
\item[(i)]  If $P\in\mathcal{C}(X,A)$ then $P\subseteq Y_k$ for some class
  $Y_k\in\mathfrak{Q}$.
\item[(ii)] There is a partial order $\triangleleft$ on $\mathfrak{Q}$ such
  that for any two distinct classes $Y', Y''\in\mathfrak{Q}$ such that
  $Y'\triangleleft Y''$ whenever there are columns $P\in Y'$ and
  $Q\in Y''$ with $P\prec Q$.
\end{itemize}
We call the classes of such a partition \emph{blocks}.  By
Lemma~\ref{lem:subset} each block $(X,A,\prec)[Y_k]$ is again an
alignment.

\begin{lemma}
  Given blocks $(X,A,\prec)[Y_k]$ and the partial order $\triangleleft$,
  there is an alignment $(X,A,\prec')$, where $\prec'$ is an an extension
  of $\prec$ defined by $P\prec' Q$ if and only if $P\prec Q$ for
  $P,Q\in Y$ for some $Y\in\mathfrak{Q}$ and $P\prec' Q$ for $P\in Y'$ and
  $Q\in Y''$ with $Y'\triangleleft Y''$.
\end{lemma}
\begin{proof}
  Each alignment block consists of the disjoint union of alignment
  column(s), thus the disjoint union of complete subgraphs. Given the
  partial order of alignment columns given by $P\prec Q$, this order is
  preserverd inside the alignment blocks $Y_k$ as each block is an
  alignment, too.  Given an alignment block $Y$ with $P\prec Q$ for
  $P,Q\in Y$ for some $Y\in\mathfrak{Q}$, one can decompose this into two
  blocks $Y'$ and $Y''$ with at least one column in each block such that
  $P\in Y'$ and $Q\in Y''$.  Based on the decomposition of $Y$ into $Y'$
  and $Y''$ one can restore the order of the alignment blocks such that
  $Y'\triangleleft Y''$ based on $Y$.  Thus, one gets the order of
  $P\prec' Q$ that is present for the alignment columns $P$ and $Q$ as well
  as for the alignment blocks $Y'$ and $Y''$.
\end{proof}

In the case of totally ordered inputs, the restriction $X_a\cap Y$ of a block
$Y$ to an input $X_a$ is an interval of $X_a$ and the columns in $Y$ form
an interval of the columns of $(X,A,<)$. Similarly, one can restrict the choice
of blocks in such a way that $\triangleleft$ just ``mirrors'' the initial
partial order, i.e., $Y'\triangleleft Y''$ if and only if $P\prec Q$ for
$P$ in $Y'$ and $Q$ in $Y''$, in which case $\prec'\;=\;\prec$ and the original
alignment is recovered by the concatenation of the blocks. In particular,
this also guarantees that valid block decompositions can be constructed for
alignments satisfying (\Dprecover).
  
Each alignment can thus be recursively decomposed into blocks. This sets
the stage for Divide-and-Conquer algorithms such as \texttt{DCA}
\cite{Stoye:97}, which cuts the sequences to be aligned into subsequences
and then concatenates the subalignments so as to optimize a global
score. In order to find the best cut-points, the algorithm recurses on
differently cut subsequences. Algorithms such as \texttt{dialign}
\cite{Morgenstern:99b} work in a conceptually similar manner but use a
bottom-up instead of a top-down approach: they first identify blocks with
high sequence conservation as ``anchors'' and recurse to construct
alignments for sequences between them.

An extreme case of the block-wise decomposition is to consider the division
of an alignment $(X,A,\prec)$ into a single maximal (or minimal) alignment
column $P$, and the rest $(X\setminus P,A',\prec)$ of the alignment. In
order for $X\setminus A\triangleleft P$ to hold, we have to ensure that
$p_a\not\prec_a q_a$ for all $p_a\in P$ and $q_a\in X\setminus P$, i.e.,
the column $P$ must entirely consist of suprema of the respective input
posets. Under this condition, we obtain a recursive column-wise
decomposition of alignments. As we shall see in the following section, this
recursion can also be used constructively.

\section{Recursive Construction}

Given a poset $(Y,\prec)$ we say that $P\subseteq Y$ is a \emph{bottom
  set} if, for all $p\in P$, every $p'\prec p$ satisfies $p'\in P$.  By
definition, the empty set, $Y$ itself, as well as the set
$\{p'\in Y| p'\preceq y\}$ for each $y\in Y$ are bottom sets. Note,
however, that $P$ also may contain points that are incomparable to all
other elements of $P$. Denote by $\sup P$ the set of \emph{suprema} of $P$,
i.e., the points such that there is no $p'\in P$ with $p\prec p'$. Clearly,
if $P$ is a bottom set and $p\in\sup P$ then $P\setminus\{p\}$ is again a
bottom set. The latter observation suggests that there is a recursive
construction for the set of alignments of $(X_1,\prec_1)$ and
$(X_2,\prec_2)$.

Denote by $\AL{A}{P}{Q}$ the set of all pairwise alignments on bottom sets
$P$ in $X_1$ and $Q$ in $X_2$. An alignment $\mathbb{A}\in \AL{A}{P}{Q}$ is
necessarily of one of three types:
\begin{itemize}
\item[(i)] $\mathbb{A}=\mathbb{A}'\VV{p}{q}$ with
  $\mathbb{A}'\in \AL{A}{P'}{Q'}$,
\item[(ii)]  $\mathbb{A}=\mathbb{A}'\VV{p}{-}$ with
  $\mathbb{A}'\in \AL{A}{P'}{Q}$, or
\item[(iii)]  $\mathbb{A}=\mathbb{A}'\VV{-}{q}$ with
  $\mathbb{A}'\in \AL{A}{P}{Q'}$,
\end{itemize}
where $P':= P\setminus\{p\}$ for $p\in\sup P$, $Q':= Q\setminus\{q\}$ for
$q\in\sup Q$, and $\AL{A}{\emptyset}{\emptyset}$ contains only the empty
alignment.

The three cases correspond to a (mis)match, insertion, and deletion.  It is
important to note that this recursion is in general not unique because the
columns extracted from $\mathbb{A}$ in consecutive steps are not
necessarily ordered relative to each other whenever $|\sup P|\ge 1$ or $|\sup Q|\ge
1$. It is, however, a proper generalization of the Needleman-Wunsch
recursion \cite{Needleman:70} for the pairwise alignment of ordered sets
(strings): If the $\prec_a$ are total orders, then $\sup P_a$ always
contains a single element, and we recover the usual Needleman-Wunsch
algorithm. In order to have a proper start and end case for the recursion
and thus DP-algorithm, it is convenient to introduce ``virtual'' source and
a sink nodes being connected to all start or end nodes of the poset,
respectively.

This idea generalizes to alignments of an arbitrary number of partial
orders in the obvious way. Denote by $\mathfrak{A}(P_1,P_2,\dots,P_N)$ the
set of all alignments where the $P_a$ are a bottom set of $(X_a,\prec_a)$.
\begin{theorem}
  Every alignment $\mathbb{A}\in \mathfrak{A}(P_1,P_2,\dots,P_N)$ is of the
  form $\mathbb{A'}\Xi$ where the alignment column $\Xi$ is a supremum
  w.r.t\ the partial order of $\prec$ of alignment columns and
  $\mathbb{A'}\in \mathfrak{A}(P'_1,P'_2,\dots,P'_N)$. The column $\Xi$
  contains in row $a$ either a gap row $a$, in which case $P'_a=P_a$, or
  $p_a\in\sup P_a$, in which case $P'_a=P_a\setminus\{p_a\}$, and does not
  entirely consist of gaps. For every column $\Upsilon$ of $\mathbb{A}'$
  we have either $\Upsilon\prec\Xi$ or $\Upsilon$ and $\Xi$ are
  incomparable.
\end{theorem}
\begin{proof}
  The $P'_a$ are again bottom sets, hence $\mathbb{A'}$ is an alignment.
  By assumption, there is a partial order on the columns $\prec$ of
  $\mathbb{A}'$.  Since every non-gap entry in $\Xi$ is a $p_a\in\sup P_a$,
  it follows that this partial order extends to $\mathbb{A}$ if and only if
  $\Xi$ is a supremum, i.e., it is either incomparable with or larger than
  any column in $\mathbb{A}'$.  Now suppose that the column $\Xi$ contains
  a $q_a\notin\sup P_a$, i.e., there is a $p_a\in X_a$ with $p_a\succ
  q_a$. Consider the column $\Upsilon$ containing $p_a$. Then either no
  partial order $\prec$ on the columns exists (contradicting that
  $\mathbb{A}'$ is an alignment), or $\Upsilon\succ\Xi$ (contradicting that
  $\Xi$ is a supremum for the alignment columns.
\end{proof}

The bottom sets are of course uniquely defined by their suprema. Clearly
$\sup P$ is an antichain, i.e., its elements are pairwisely
incomparable. Conversely, every antichain $U$ in $(X_a,\prec_a)$ uniquely
defines a bottom set $P:=\{p\in X_a| p\preceq U\}$. It is obvious therefore
that for two bottom sets $P$ and $Q$ it holds that $P=Q$ if and only if
$\sup P = \sup Q$. Hence there is a 1-1 correspondence between the
antichains of a partial order and their bottom sets. The recursion in the
theorem can be written in terms of the antichains of the $(X_a,\prec_a)$.

In order to capture the more restrictive notion of alignments satisfying
(\Dprecover) the recursion has to be modified in a such a way that for
every (mis)match between two rows it can be ensured that all previously
formed columns are either comparable in both rows or incomparable in both
rows. This is non-trival because this information is not purely local. For
ease of discussion, we only consider the case of aligning two posets. There
are at least two strategies to maintain this information.

Attempting to construct a similar recursion as in the (\Dpcross) case, one
could store with each pair $P\in X_1$ and $Q\in X_2$ also all the set
$\mathcal{M}$ of all matchings $\VV{p}{q}$ ``to the right'' of $P$ and $Q$,
i.e., $p\in X_1\setminus P$ and $q\in X_1\setminus Q$. Then every allowed
matching/column $\VV{p'}{q'}$, $p'\in\sup P$ and $q'\in\sup Q$ must
satisfy: for all $\VV{p}{q}\in\mathcal{M}$ holds: either $p'\prec p$ and
$q'\prec q$, or both $p',p$ and $q',q$ are incomparable. Every such pair
can be appended to $\mathcal{M}$, with corresponding updates
$P\to P\setminus\{p'\}$ and $Q\to Q\setminus\{q'\}$. Insertions and
deletions of course only require the removal of either $p'$ from $P$ or
$q'$ from $Q$, respectively. Initially, $P=X_1$, $Q=X_2$, and
$\mathcal{M}=\emptyset$. Every set of valid partial alignments is
characterized by a triple $(P,Q,\mathcal{M})$.

An alternative approach is to store instead for each $p\in P$ and $q\in Q$
also the sets $c_Q(p)$ and $c_P(q)$ that can form matches $\VV{p}{q'}$,
$q'\in c_Q(p)$ and $\VV{p'}{q}$, $p\in c_P(q)$, respectively. Initially, we
have $P=X_1$, $Q=X_2$, $c_Q(p)=Q$ for all $p\in P$ and $c_P(q)=P$ for all
$q\in Q$.  Whenever an alignment is continued with a (mis)match
$\VV{p}{q}$, $p\in \sup P$, $q\in \sup Q$, we have to remove all candidates
from $c_P(q')$ and $c_Q(p')$ that are inconsistent with $\VV{p}{q}$. That
is: if $q'\prec q$, then
$c_P(q')\leftarrow \{p'\in c_P(q')\vert p'\prec p\}$. If $q$ and $q'$ and
incomparable, then
$c_P(q')\leftarrow \{p'\in c_P(q')\vert p',p \textrm{ incomparable}\}$. The
$c_Q(p')$ are updated correspondingly.  In the case of an insertion
$\VV{p}{-}$, we only need to remove $p$ from $f_P(q')$, $q'\in Q$.
Similarly, $\VV{-}{q}$ implies that $q$ has to be removed from the
$f_Q(p')$ for all $p'\in P$. We suspect that an encoding of alignment sets
of the form $(P,f_Q:P\to 2^P; Q, f_P:Q\to 2^P)$ will be efficient if the
poset has only small antichains. A more detailed analysis of this kind of
recursive construction from the point of view of algorithmic efficiency
will be considered elsewhere.

The POA algorithm \cite{Lee:02} computes the alignment of two posets satisfying
(\Dprecover), albeit with the restriction that one of the two inputs is totally
ordered. This removes all ambiguities in the totally ordered po-set and implies
that, given any match $\VV{u}{v}$ in the alignment, all preceeding matches
$\VV{u'}{v'}$ satisfy $v'<v$ in the totally ordered set and thus $u'$ must be a
predecessor of $u$. The alignment thus must follow a single path in the Hasse
diagram of the unrestricted input poset.

\section{Pairwise Alignments as Relations} 

Pairwise alignments have a particularly simple structure. In particular,
they are bipartite (undirected) graphs, and hence can be regarded equivalently
as symmetric binary relations $R\subseteq X_1\times X_2$. More precisely,
we can identify a relation $R$ with an undirected graph with vertex set
$X_1\dot\cup X_2$ and (undirected) edges $\{x_1,x_2\}$ whenever
$(x_1,x_2)\in R$. We write this graph as $(X_1\dot\cup X_2,R)$. 

Relations have a natural composition. For $R\subseteq X\times Y$ and
$S\subseteq Y\times Z$ is is defined by 
\begin{equation}
  (x,z)\in S\circ R \quad\textrm{iff}\quad
  \exists y\in Y \textrm{ s.t. } (x,y)\in R\textrm{ and } (y,z) \in S
\end{equation}

In the following we will be interested in the following properties of
  binary relations:
\begin{itemize}
\item[(M)]
  $(x,y)\in R$ and $(x,z)\in R$ implies $y=z$ and 
  $(x,z)\in R$ and $(y,z)\in R$ implies $x=y$.
\item[(P')] There is a partial order $\prec$ on $R$ such that $u\prec_1 x$
  or $v\prec_2 y$ implies $(u,v)\prec(x,y)$.
\item[(P)] If $(x_1,y_1)\in R$ and $(x_2,y_2)\in R$ then $x_1\prec x_2$ if
  and only if $y_1\prec y_2$.
\end{itemize}

\begin{lemma}
  The composition of two binary relations satisfying (M) and (P) is again
  a binary relation satisfing (M) and (P).
\end{lemma}
\begin{proof}
  Suppose $(x,z)\in R\circ S$. Then there is $y$ such that
  both $(x,y)\in R$ and $(y,z)\in S$. By (M), there is no other $y'\ne y$
  with $(x,y')\in R$ and no $z'\ne z$ such that $(y',z')\in S$, hence in
  particular there is no $z'\ne z$ such that $(x,z')\in R\circ S$.
  Analogously, one argues that there is no $x'\ne x$ such that
  $(x',z)\in R\circ S$. Thus $R\circ S$ again satisfies (M). 

  Suppose $(x_1,z_1), (x_2,z_2) \in R\circ S$. By (M) there are unique
  vertices $y_1$ and $y_2$ such that $(x_1,y_1),(x_2,y_2)\in R$ and
  $(y_1,z_1),(y_2,z_2)\in S$, respectively. Now suppose $x_1\prec_1 x_2$.
  Then (P) implies $y_1 \prec_2 y_2$, and using (P) again yields
  $z_1 \prec_3 z_2$. Starting from $z_1 \prec_3 z_2$, the same argument
  yields $z_1\prec_1 z_2$.  Conversely, suppose
  $(x_1,z_1), (x_2,z_2) \in R\circ S$ and $x_1,x_2$ are incomparable. By
  (M) there are unique vertices $y_1$ and $y_2$ with
  $(x_1,y_1), (x_2,y_2) \in R$ and $(y_1,z_1), (y_2,z_2) \in S$, for which
  (P) now implies that they are incomparable. Using the same argument again
  shows that that $z_1$ and $z_2$ also must be incomparable. Hence
  concatenation preserves not only the relative order but also
  comparability, i.e., $R\circ S$ again satisfies (P). 
\end{proof}

It is easy to see that Axiom (P') is in general not preserved under
concatenation: Requiring only (P') allows the intermediate vertices $y_1$
and $y_2$ to be incomparable. Hence it is possible in this scenario to
have $x_1\prec_1 x_2$, incomparable vertices $y_1$ and $y_2$, and
$z_2\prec_3 z_1$ with $(x_1,y_1),(x_2,y_2)\in R$ and
$(y_1,z_1),(y_2,z_2)\in S$ while the concatenation violates the (P').

A relation satisfying (M) and (P') can easily be extended to an alignment
$(X_1\cup X_2,R)$ considering each edge $(x_1,y_1)$ and considering all
unmatched positions, i.e., every $\{x'\}$ such that there is no
$y\in X_2 (x',y)$ and every $\{y'\}$ such that there is no
$x\in X_1 (x,y')$ as alignment columns.  The relative order of these
columns is inherited from the partial order $(X_1,\prec_1)$ and
$(X_2,\prec_2)$. 

\begin{lemma} 
  Every pairwise alignment satisfying {\rm(\Dpsub)}, {\rm(\Dpone)},
  {\rm(\Dporder)}, and {\rm(\Dpcross)} can be written as an extension of
  the a binary relation $R\subseteq X_1\times X_2$ satisfying {\rm(M)} and
  {\rm(P')}. Conversely, every binary relation $R\subseteq X_1\times X_2$
  satisfying {\rm(M)} and {\rm(P')} gives rise to an alignment satisfying
  {\rm(\Dpsub)}, {\rm(\Dpone)}, {\rm(\Dporder)}, and {\rm(\Dpcross)}.
\end{lemma}
\begin{proof}
  By definition, all edges are incident to one vertex in $X_1$ and one
  vertex in $X_2$, thus the graph is a bipartite matching. Condition (M) is
  therefore equivalent to (\Dpsub) and (\Dpone) for the case of two input
  posets. Axiom (\Dporder) implies the ordering required by (P') as well as
  its extension to the in/del columns. (\Dpcross) and (P') equivalently
  guarantee the existence of the partial order on the columns that satisfy
  (\Dporder).
\end{proof}

\begin{lemma} 
  Every pairwise alignment satisfying {\rm(\Dprecover)} corresponds to a
  binary relation $R\subseteq X_1\times X_2$ satisfying {\rm(M)} and
  {\rm(P)}.
\end{lemma}
\begin{proof}
  Axiom (\Dprecover) simplifies to (P) in the case of only two inputs. The
  existence of the required partial order on the set of all columns is
  guaranteed by Lemma~\ref{lem:po2}.
\end{proof}

This suggests that the more restrictive condition (\Dprecover) may be a more
natural condition for defining alignments of partially ordered sets.  As a
down-side, however, it seems that there is no convenient recursive
construction of the search space similar to the dynamic programming
approaches for sequence alignment. Instead, it seems more natural to treat
this class of alignment problems as maximum induced subgraph problems.

Composition of binary relations is a powerful tool to construct multiple
alignments. Suppose we are given a set of posets $(X_a,\prec_a)$ and a set
$\mathcal{R}$ of pairwise relations satisfying (M) and (P) such that the
graph representation of $\mathcal{R}$ is tree, then there is a unique
multiple alignment satisfying (\Dprecover) obtained as the transitive closure of the
graph on $X$ with edges defined by the $R\in\mathcal{R}$.  However, not
every alignment can be represented in this manner. As a
simple counterexample consider the alignment of the three sequences\\
\begin{verbatim}
  a  A-C         a  A-C                  a  A-C
  b  -BC         b  -BC     b  -BC       b  -BC
  c  AB-                    c  AB-       c A-B-
\end{verbatim}
where the composition of any two pairwise alignments gives rise to two
different columns for in/del columns of the pairwise components, in the
example of two \texttt{A} entries. On the other hand the progressive
approach, in which sequence \texttt{c} is aligned to the pairwise alignment
of \texttt{a} and \texttt{b} yields the example alignment. In fact,
Lemma~\ref{lem:quotient} implies that in principle every alignment can be
obtained by a progressive alignment scheme.

If $\mathcal{R}$ contains cycles, then there is no guarantee that the
transitive closure $\widehat{A}$ of $\bigcup_{R\in \mathcal{R}}R$ is an
alignment: In general, both conditions (\Dpsub) and (\Dpone) will be
violated. So-called \emph{transitive alignment} approaches deliberately
accept this at an intermediate stage. Various heuristics can be used to
remove superfluous edges from the graph $(X,\widehat{A})$, that is they
construct a subgraph $(X,A)$, $A\subseteq\widehat{A}$ that again satisfies
all conditions of a valid alignment.

\section{Discussion} 

An interesting idea that follows quite naturally from the discussion above
is a general approach towards graph comparison for sets of graphs: it seems
natural to generalize the idea of progressive alignments in the following
manner: (1) Given two graphs $G_1$ and $G_2$ and a common induced subgraph
$H$ (strictly speaking together with an embedding of $H$ into $G_1$ and
$G_2$) the graph defined by identifing the copies of $H$ in $G_1$ and $G_2$
can be thought of as pairwise alignment. If $G_1$ and $G_2$ have vertex
labels $\alpha_i:V(G_i)\to A_i$, $i=1,2$ for some alphabets $A_i$, on
labels $G_1 \bullet_H G_2$ with label pairs ($\alpha_1(x),\alpha_2(x)$) for
$x\in V(H)$, $(\alpha_1(x),-)$ for $x\in V(G_1\setminus H)$ and
$(-,\alpha_2)$ for $x\in V(G_2\setminus H)$. Naturally, an optimization
criterion such as ``maximal common induced subgraph'' will be used in
practice. Since $G_1 \bullet_H G_2$ is again a (labeled) graph, the
procedure can be repeated e.g.\ along a line of guidetrees. This gives
raise to a natural notion of a multiple alignment of graphs $G_a$ with
vertex sets $V(G_a)$ and edge sets $E(G_a)$. Let $X=\dot\bigcup V(G_a)$ and
$A$ be a set of undirected edges on $X$ and let $\mathcal{C}(X,A)$ be the set
of connected components of the graph $(X,A)$. Then $(X,A,E^*)$ is a
multiple alignment of the graph $G_a$, where $E^*$ denotes the set of edges
on $\mathcal{C}(X,A)$.
\begin{itemize}
\item[(G1)] $Q\in\mathcal{C}(X,A)$ is complete subgraph of $(X,A)$.
\item[(G2)] If $(a,i)\in Q$ and $(a,j)\in Q$, then $i=j$.
\item[(G3)] If $(a,i)\in P$, $(a,j)\in Q$ for some
  $P,Q\in \mathcal{C}(X,A)$ and $((a,i),(a,j))\in E(G_a)$ then
  $(P,Q)\in E^*$
\item[(G5)] If $(P,Q)\in E$, $(a,i)\in P$, and $(a,j)\in Q$ then
  $((a,i),(a,j))\in E^*$
\end{itemize}
The graph $(\mathcal{C}(X,A),E^*)$ can be constructed as the quotient graph
$(X,A\cup\bigcup_a E(G_a))/\mathcal{C}(X,A)$ obtained by adding in all the
edges of $G_a$ and collapsing all columns (connected components) of $(X,A)$
to a single vertex.  A plausible generalization of (\Dpcross) might be to
require
\begin{itemize}
\item[(G4)] If $(P,Q)\in E^*$ then there is a row $a$ with $(a,i)\in P$,
  $(a,j)\in Q$ and $((a,i),(a,j))\in E(G_a)$,
\end{itemize}
i.e., (G3) completely determines the edges between alignment columns. In
this setting (G4) does not impose additional conditions on the columns.
However, if both the input graphs $G_a$ and the alignment graph
$(\mathcal{C}(X,A),E^*)$ are restricted to particular graph classes, such
constraints appear. The graphs of partially ordered sets, i.e., the
transitive acyclic digraphs discussed at length in the previous sections,
of course, serve as a non-trivial example.

It is important to note the graph alignment in the sense used here --
namely requiring a matching between vertices and notion of structural
congruence between the alignment and its consitutent graphs -- are more
restrictive than some concepts of ``graph alignments'' discussed in the
literature. In particular, we make a sharp distinction here between ``graph
alignments'' and various approaches of comparison by means of graph
editing, see e.g.\ \cite{EmmertStreib:16} for a recent review.

The example of graph alignments and oriented tree alignments suggests to
consider an even broader class of structures: Given a (finite) collection
of sets $X_a$, each endowed with a set of relational structures (or more
general set systems), we may ask for collections of partial maps between
any pair of them that satisfy (G1) and (G2), i.e., define a partition on
$\bigcup_a X_a$ such that each class contains at most one element of each
$X_a$ and the relation (or set system) structure is preserved in a sense
similar to conditions (G3) and (G5) above.  Such constructions are of
practical interest e.g.\ for alignments of oriented trees (or forests)
\cite{Jiang:95,Hoechsmann:04,Berkemer:17}, where two distinct partial
orders are defined, one capturing the order implied by parent-child
relationship and another one representing the relative order of siblings.
Alignments of ordered trees preserve both partial orders, since the
alignment is defined as an ordered tree on the columns such that each
ordered input tree (with vertices $X_a$) is obtained as a restriction to
exactly the columns in which row $a$ does not have a gap
entry. Intuitively, this seems to require that (1) a super-object $G$
exists for a pair of objects $G_1$ and $G_2$ such that $G_1$ and $G_2$ can
be obtained as projections and (2) an intersection of the embeddings of
$G_1$ and $G_2$ into $G$ defines an induced sub-object $H$ common to $G_1$
and $G_2$ is well defined. While the super-object corresponds to the
alignments, the sub-object takes on the role of matches in the alignment.
A similar notion of alignment is used in computational biology for RNA
structures, where base pairs need to be preserved in addition the total order
of the input sequences \cite{Moehl:10}. Here, however, only consistency
similar in flavor to (\Dpcross) is enforced, suggesting that it may be of
interest to relax the requirement of \emph{induced} sub-objects.

The recursive formulation of the poset alignments is an extension of the
well-known Needleman-Wunsch alignment algorithm. Beyond many
implementations of the Needleman-Wunsch algorithm, the implementation based
on \texttt{ADPfusion} (Algebraic Dynamic Programming with compile-time fusion
of grammar and algebra) \cite{HOE:2012} is designed in a
way to be extendable to different scoring functions, problem descriptions,
and data structures \cite{Hoener:15a}.  Future work thus will include the
adaptation of the \texttt{ADPfusion} framework written in a functional
language (Haskell) to the data structure of posets. Earlier adaptations of
the Needleman-Wunsch algorithm to trees, forests and sets already exist
\cite{Berkemer:17,HOE:PRO:2015}.

It may also be possible to implement the poset alignment algorithm for the
(\Dprecover)-notion of alignments in a way similar to the graph alignment algorithm
above, i.e., starting from a maximal induced common subgraph that is then
extended. Depending on the structure of the posets, this might be more
efficient than the recursive DP algorithm where additional information has
to be stored and updated in each step.

Finding maximal induced common subgraphs is well known to be a NP-complete problem.
Nevertheless, DP algorithms have been devised for restricted settings such as
planar graphs \cite{Fomin:11}. These proved practical for moderate size
problems even though their resource requirements still scale exponentially.

For general graphs, there exist algorithms to detect common subgraphs
\cite{Akutsu:93}. However, as the problem is NP-complete, the problem can only
be solved for small instances of the input structures in a reasonable amount of
time. For small graphs such as representations of small chemical molecules, the
DP algorithm might be able to solve the maximal common subgraph problem as
described in \cite{Akutsu:93}. Here, the DP algorithm is based on the analogous
version for trees where the vertex degree has to be bounded in order to find a
solution in a reasonable time frame. The algorithm divides the input structures
in (overlapping) biconnected components and tries to find the best match
between both input graphs preserving the order of the biconnected components of
the original graph.

Finally, it seems natural to consider multiple alignments at a more abstract
level: In order to properly define them, it seems sufficient that (induced)
sub-objects can be used to ``glue together'' two (and recursively more) objects
in such a way that the resulting super-object projects down to the given inputs.
It is natural to ask how such structures can be characterized in the language
of category theory. Is there an interesting class of categories that admit
well-defined alignments objects, and do the resulting alignments themselves
from categories with useful properties?

\section*{Acknowledgements}

This work was supported in part by the German Academic Exchange Service
(DAAD), proj.no.\ 57390771.

\bibliographystyle{plainnat}
\bibliography{aligncomposition}

\begin{thebibliography}{58}
\providecommand{\natexlab}[1]{#1}
\providecommand{\url}[1]{\texttt{#1}}
\expandafter\ifx\csname urlstyle\endcsname\relax
  \providecommand{\doi}[1]{doi: #1}\else
  \providecommand{\doi}{doi: \begingroup \urlstyle{rm}\Url}\fi

\bibitem[Akutsu(1993)]{Akutsu:93}
Tatsuya Akutsu.
\newblock A polynomial time algorithm for finding a largest common subgraph of
  almost trees of bounded degree.
\newblock \emph{IEICE transactions on fundamentals of electronics,
  communications and computer sciences}, 76\penalty0 (9):\penalty0 1488--1493,
  1993.

\bibitem[Altschul et~al.(1997)Altschul, Madden, Sch{\"a}ffer, Zhang, Zhang,
  Miller, and Lipman]{Altschul:97}
Stephen~F Altschul, Thomas~L Madden, Alejandro~A Sch{\"a}ffer, Jinghui Zhang,
  Zheng Zhang, Webb Miller, and David~J Lipman.
\newblock Gapped {BLAST} and {PSI-BLAST}: a new generation of protein database
  search programs.
\newblock \emph{Nucleic Acids Res.}, 25:\penalty0 3389--3402, 1997.
\newblock \doi{10.1093/nar/25.17.3389}.

\bibitem[Baichoo and Ouzounis(2017)]{Baichoo:17}
Shakuntala Baichoo and Christos~A. Ouzounis.
\newblock Computational complexity of algorithms for sequence comparison,
  short-read assembly and genome alignment.
\newblock \emph{Biosystems}, 156/157:\penalty0 72--85, 2017.
\newblock \doi{10.1016/j.biosystems.2017.03.003}.

\bibitem[Berkemer et~al.(2017)Berkemer, H{\"o}ner~zu Siederdissen, and
  Stadler]{Berkemer:17}
Sarah~J Berkemer, Christian H{\"o}ner~zu Siederdissen, and Peter~F Stadler.
\newblock Algebraic dynamic programming on trees.
\newblock \emph{Algorithms}, 10:\penalty0 135, 2017.
\newblock \doi{10.3390/a10040135}.

\bibitem[Bhattacharya et~al.(2018)Bhattacharya, Blasi, Croft, Cysouw, Hruschka,
  Maddieson, M{\"u}ller, Retzlaff, Smith, Stadler, Starostin, and
  Youn]{Bhattacharya:18a}
Tanmoy Bhattacharya, Damian Blasi, William Croft, Michael Cysouw, Daniel
  Hruschka, Ian Maddieson, Lydia M{\"u}ller, Nancy Retzlaff, Eric Smith,
  Peter~F. Stadler, George Starostin, and Hyejin Youn.
\newblock Studying language evolution in the age of big data.
\newblock \emph{J. Language Evol.}, 3:\penalty0 94--129, 2018.
\newblock \doi{10.1093/jole/lzy004}.

\bibitem[Bonizzoni and Della~Vedova(2001)]{Bonizzoni:01}
Paola Bonizzoni and Gianluca Della~Vedova.
\newblock The complexity of multiple sequence alignment with {SP}-score that is
  a metric.
\newblock \emph{Theor. Comp. Sci.}, 259:\penalty0 63--79, 2001.
\newblock \doi{10.1016/S0304-3975(99)00324-2}.

\bibitem[Bunke(1997)]{Bunke:97}
H.~Bunke.
\newblock On a relation between graph edit distance and maximum common
  subgraph.
\newblock \emph{Pattern Recognition Letters}, 18:\penalty0 689--694, 1997.
\newblock \doi{10.1016/S0167-8655(97)00060-3}.

\bibitem[Carrillo and Lipman(1988)]{Carillo:88}
Humberto Carrillo and David Lipman.
\newblock The multiple sequence alignment problem in biology.
\newblock \emph{SIAM J. Appl. Math.}, 48:\penalty0 1073--1082, 1988.
\newblock \doi{10.1137/0148063}.

\bibitem[Cysouw and Jung(2007)]{Cysouw:07}
Michael Cysouw and Hagen Jung.
\newblock Cognate identification and alignment using practical orthographies.
\newblock In \emph{Proceedings of Ninth Meeting of the ACL Special Interest
  Group in Computational Morphology and Phonology}, pages 109--116. Association
  for Computational Linguistics, 2007.
\newblock URL \url{https://www.aclweb.org/anthology/W/W07/W07-1314.pdf}.

\bibitem[Do et~al.(2005)Do, Mahabhashyam, Brudno, and Batzoglou]{Do:05}
Chuong~B Do, Mahathi~SP Mahabhashyam, Michael Brudno, and Serafim Batzoglou.
\newblock {ProbCons}: Probabilistic consistency-based multiple sequence
  alignment.
\newblock \emph{Genome Res.}, 15:\penalty0 330--340, 2005.
\newblock \doi{10.1101/gr.2821705}.

\bibitem[Durbin et~al.(1998)Durbin, Eddy, Krogh, and Mitchison]{Durbin:98}
Richard Durbin, Sean~R Eddy, Anders Krogh, and Graeme Mitchison.
\newblock \emph{Biological sequence analysis: probabilistic models of proteins
  and nucleic acids}.
\newblock Cambridge University Press, Cambridge, UK, 1998.

\bibitem[Edgar and Batzoglou(2006)]{Edgar:06}
R~C Edgar and S~Batzoglou.
\newblock Multiple sequence alignment.
\newblock \emph{Curr Opin Struct Biol}, 16:\penalty0 368--373, 2006.
\newblock \doi{10.1016/j.sbi.2006.04.004}.

\bibitem[Edgar(2004)]{Edgar:04}
Robert~C Edgar.
\newblock {MUSCLE}: multiple sequence alignment with high accuracy and high
  throughput.
\newblock \emph{Nucleic Acids Res.}, 32:\penalty0 1792--1797, 2004.
\newblock \doi{10.1093/nar/gkh340}.

\bibitem[Elias(2006)]{Elias:06}
Isaac Elias.
\newblock Settling the intractability of multiple alignment.
\newblock \emph{J. Comp. Biol.}, 13:\penalty0 1323--1339, 2006.
\newblock \doi{10.1089/cmb.2006.13.1323}.

\bibitem[Emmert-Streib et~al.(2016)Emmert-Streib, Dehmer, and
  Shi]{EmmertStreib:16}
Frank Emmert-Streib, Matthias Dehmer, and Yongtang Shi.
\newblock Fifty years of graph matching, network alignment and network
  comparison.
\newblock \emph{Information Sci.}, 346/347:\penalty0 180--197, 2016.
\newblock \doi{10.1016/j.ins.2016.01.074}.

\bibitem[Fagin et~al.(2003)Fagin, Kumar, and Sivakumar]{Fagin:03}
Ronald Fagin, Ravi Kumar, and D.~Sivakumar.
\newblock Comparing top-$k$ lists.
\newblock \emph{SIAM J.\ Discr.\ Math.}, 17:\penalty0 134--160, 2003.
\newblock \doi{10.1137/S0895480102412856}.

\bibitem[Feng and Doolittle(1987)]{Feng:87}
Da-Fei Feng and Russell~F Doolittle.
\newblock Progressive sequence alignment as a prerequisite to correct
  phylogenetic trees.
\newblock \emph{J. Mol. Evol.}, 25:\penalty0 351--360, 1987.
\newblock \doi{10.1007/BF02603120}.

\bibitem[Fomin et~al.(2011)Fomin, Todinca, and Villanger]{Fomin:11}
Fedor~V. Fomin, Ioan Todinca, and Yngve Villanger.
\newblock Exact algorithm for the maximum induced planar subgraph problem.
\newblock In Camil Demetrescu and Magn{\'u}s~M. Halld{\'o}rsson, editors,
  \emph{Proceedings of the 19th European conference on Algorithms}, volume 6942
  of \emph{Lecture Notes Comp. Sci.}, pages 287--298, Berlin, Heidelberg, 2011.
  Springer-Verlag.

\bibitem[Gotoh(1982)]{Gotoh:82}
O~Gotoh.
\newblock An improved algorithm for matching biological sequences.
\newblock \emph{J. Mol. Biol.}, 162:\penalty0 705--708, 1982.
\newblock \doi{10.1016/0022-2836(82)90398-9}.

\bibitem[Gotoh(1986)]{Gotoh:86}
O.~Gotoh.
\newblock Alignment of three biological sequences with an efficient traceback
  procedure.
\newblock \emph{J. theor. Biol.}, 121:\penalty0 327--337, 1986.
\newblock \doi{10.1016/S0022-5193(86)80112-6}.

\bibitem[Grabherr et~al.(2010)Grabherr, Russell, Meyer, Mauceli, Alf{\"o}ldi,
  Di~Palma, and Lindblad-Toh]{Grabherr:10}
Manfred~G. Grabherr, Pamela Russell, Miriah Meyer, Evan Mauceli, Jessica
  Alf{\"o}ldi, Federica Di~Palma, and Kerstin Lindblad-Toh.
\newblock Genome-wide synteny through highly sensitive sequence alignment:
  \textit{Satsuma}.
\newblock \emph{Bioinformatics}, 26:\penalty0 1145--1151, 2010.
\newblock \doi{10.1093/bioinformatics/btq102}.

\bibitem[Grasso and Lee(2004)]{Grasso:04}
Catherine Grasso and Christopher Lee.
\newblock Combining partial order alignment and progressive multiple sequence
  alignment increases alignment speed and scalability to very large alignment
  problems.
\newblock \emph{Bioinformatics}, 20:\penalty0 1546--1556, 2004.
\newblock \doi{10.1093/bioinformatics/bth126}.

\bibitem[H{\"o}chsmann et~al.(2004)H{\"o}chsmann, Voss, and
  Giegerich]{Hoechsmann:04}
Michael H{\"o}chsmann, Bj{\"o}rn Voss, and Robert Giegerich.
\newblock Pure multiple {RNA} secondary structure alignments: a progressive
  profile approach.
\newblock \emph{IEEE/ACM Trans. Comp. Biol. Bioinf.}, 1:\penalty0 53--62, 2004.
\newblock \doi{10.1109/TCBB.2004.11}.

\bibitem[H{\"o}ner~zu Siederdissen(2012)]{HOE:2012}
Christian H{\"o}ner~zu Siederdissen.
\newblock Sneaking around concat{M}ap: Efficient combinators for dynamic
  programming.
\newblock In \emph{{Proceedings of the 17th ACM SIGPLAN international
  conference on Functional programming}}, ICFP '12, pages 215--226, New York,
  NY, USA, 2012. ACM.
\newblock ISBN 978-1-4503-1054-3.
\newblock \doi{10.1145/2364527.2364559}.
\newblock URL \url{http://www.bioinf.uni-leipzig.de/Software/gADP/}.

\bibitem[H{\"o}ner~zu Siederdissen et~al.(2015{\natexlab{a}})H{\"o}ner~zu
  Siederdissen, Hofacker, and Stadler]{Hoener:15a}
Christian H{\"o}ner~zu Siederdissen, Ivo~L. Hofacker, and Peter~F. Stadler.
\newblock Product grammars for alignment and folding.
\newblock \emph{IEEE/ACM Trans. Comp. Biol. Bioinf.}, 12:\penalty0 507--519,
  2015{\natexlab{a}}.
\newblock \doi{10.1109/TCBB.2014.2326155}.

\bibitem[H{\"o}ner~zu Siederdissen et~al.(2015{\natexlab{b}})H{\"o}ner~zu
  Siederdissen, Prohaska, and Stadler]{HOE:PRO:2015}
Christian H{\"o}ner~zu Siederdissen, Sonja~J. Prohaska, and Peter~F. Stadler.
\newblock Algebraic dynamic programming over general data structures.
\newblock \emph{BMC Bioinformatics}, 16 Suppl 19:\penalty0 S2,
  2015{\natexlab{b}}.
\newblock \doi{10.1186/1471-2105-16-S19-S2}.

\bibitem[Jiang et~al.(1995)Jiang, Wang, and Zhang]{Jiang:95}
Tao Jiang, Lusheng Wang, and Kaizhong Zhang.
\newblock Alignment of trees -- an alternative to tree edit.
\newblock \emph{Theor. Comp. Sci.}, 143:\penalty0 137--148, 1995.
\newblock \doi{10.1016/0304-3975(95)80029-9}.

\bibitem[Just(2001)]{Just:01}
Winfried Just.
\newblock Computational complexity of multiple sequence alignment with
  {SP}-score.
\newblock \emph{J. Comp. Biol.}, 8:\penalty0 615--623, 2001.
\newblock \doi{10.1089/106652701753307511}.

\bibitem[Katoh et~al.(2005)Katoh, Kuma, Toh, and Miyata]{Katoh:05}
Kazutaka Katoh, Kei-ichi Kuma, Hiroyuki Toh, and Takashi Miyata.
\newblock {MAFFT} version 5: improvement in accuracy of multiple sequence
  alignment.
\newblock \emph{Nucleic Acids Res.}, 33:\penalty0 511--518, 2005.
\newblock \doi{10.1093/nar/gki198}.

\bibitem[Kececioglu(1993)]{Kececioglu:93}
J.~D. Kececioglu.
\newblock The maximum weight trace problem in multiple sequence alignment.
\newblock In \emph{Proceedings of the 4th Symposium on Combinatorial Pattern
  Matching}, volume 684 of \emph{Lecture Notes Comp. Sci.}, pages 106--119,
  Berlin, 1993. Springer.

\bibitem[Kececioglu and Starrett(2004)]{Kececioglu:04}
John Kececioglu and Dean Starrett.
\newblock Aligning alignments exactly.
\newblock In Philip~E. Bourne and Dan Gusfield, editors, \emph{Proceedings of
  the 8th ACM Conference on Research in Computational Molecular Biology
  {(RECOMB)}}, pages 85--96, New York, NY, 2004. ACM.
\newblock \doi{10.1145/974614.974626}.

\bibitem[Konagurthu et~al.(2004)Konagurthu, Whisstock, and
  Stuckey]{Konagurthu:04}
A.~S. Konagurthu, J.~Whisstock, and P.~J. Stuckey.
\newblock Progressive multiple alignment using sequence triplet optimization
  and three-residue exchange costs.
\newblock \emph{J. Bioinf. and Comp. Biol.}, 2:\penalty0 719--745, 2004.
\newblock \doi{10.1142/S0219720004000831}.

\bibitem[Kondrak(2000)]{Kondrak:00}
Grzegorz Kondrak.
\newblock A new algorithm for the alignment of phonetic sequences.
\newblock In \emph{Proceedings of NAACL 2000 1st Meeting of the North American
  Chapter of the Association for Computational Linguistics}, pages 288--295,
  San Francisco, CA, USA, 2000. Morgan Kaufmann Publishers Inc.
\newblock \doi{10.1.1.19.9698}.
\newblock URL \url{http://aclweb.org/anthology/A00-2038}.

\bibitem[Kruspe and Stadler(2007)]{Kruspe:07a}
Matthias Kruspe and Peter~F. Stadler.
\newblock Progressive multiple sequence alignments from triplets.
\newblock \emph{BMC Bioinformatics}, 8:\penalty0 254, 2007.
\newblock \doi{10.1186/1471-2105-8-254}.

\bibitem[Larkin et~al.(2007)Larkin, Blackshields, Brown, Chenna, McGettigan,
  McWilliam, Valentin, Wallace, Wilm, Lopez, Thompson, Gibson, and
  Higgins]{Larkin:07}
Mark~A Larkin, Gordon Blackshields, N~P Brown, R~Chenna, Paul~A McGettigan,
  Hamish McWilliam, Franck Valentin, Iain~M Wallace, Andreas Wilm, Rodrigo
  Lopez, J~D Thompson, T~J Gibson, and D.~G. Higgins.
\newblock {Clustal W} and {Clustal X} version 2.0.
\newblock \emph{Bioinformatics}, 23:\penalty0 2947--2948, 2007.
\newblock \doi{10.1093/bioinformatics/btm404}.

\bibitem[Lee(2003)]{Lee:03}
Christopher Lee.
\newblock Generating consensus sequences from partial order multiple sequence
  alignment graphs.
\newblock \emph{Bioinformatics}, 19:\penalty0 999--1008, 2003.
\newblock \doi{10.1093/bioinformatics/btg109}.

\bibitem[Lee et~al.(2002)Lee, Grasso, and Sharlow]{Lee:02}
Christopher Lee, Catherine Grasso, and Mark~F. Sharlow.
\newblock Multiple sequence alignment using partial order graphs.
\newblock \emph{Bioinformatics}, 18:\penalty0 452--464, 2002.
\newblock \doi{10.1093/bioinformatics/18.3.452}.

\bibitem[Lipman et~al.(1989)Lipman, Altschul, and Kececioglu]{Lipman:89}
David~J Lipman, Stephen~F Altschul, and John~D Kececioglu.
\newblock A tool for multiple sequence alignment.
\newblock \emph{Proc. Natl. Acad. Sci. USA}, 86:\penalty0 4412--4415, 1989.
\newblock \doi{10.1073/pnas.86.12.4412}.

\bibitem[Malde and Furmanek(2013)]{Malde:13}
Ketil Malde and Tomasz Furmanek.
\newblock Increasing sequence search sensitivity with transitive alignments.
\newblock \emph{PloS one}, 8:\penalty0 e54422, 2013.
\newblock \doi{10.1371/journal.pone.0054422}.

\bibitem[Manthey(2003)]{Manthey:03}
Bodo Manthey.
\newblock Non-approximability of weighted multiple sequence alignment.
\newblock \emph{Theor. Comp. Sci.}, 296:\penalty0 179--192, 2003.
\newblock \doi{10.1007/3-540-44679-6_9}.

\bibitem[M{\"o}hl et~al.(2010)M{\"o}hl, Will, and Backofen]{Moehl:10}
Mathias M{\"o}hl, Sebastian Will, and Rolf Backofen.
\newblock Lifting prediction to alignment of {RNA} pseudoknots.
\newblock \emph{J Comput Biol.}, 17:\penalty0 429--442, 2010.
\newblock \doi{10.1089/cmb.2009.0168}.

\bibitem[Morgenstern(1999)]{Morgenstern:99b}
Burkhard Morgenstern.
\newblock {DIALIGN 2}: improvement of the segment-to-segment approach to
  multiple sequence alignment.
\newblock \emph{Bioinformatics}, 15:\penalty0 211--218, 1999.
\newblock \doi{10.1093/bioinformatics/15.3.211}.

\bibitem[Morgenstern et~al.(1999)Morgenstern, Stoye, and Dress]{Morgenstern:99}
Burkhard Morgenstern, Jens Stoye, and Andreas W.~M. Dress.
\newblock Consistent equivalence relations: a set-theoretical framework for
  multiple sequence alignments.
\newblock Technical report, University of Bielefeld, FSPM, 1999.

\bibitem[Needleman and Wunsch(1970)]{Needleman:70}
Saul~B Needleman and Christian~D Wunsch.
\newblock A general method applicable to the search for similarities in the
  amino acid sequence of two proteins.
\newblock \emph{J. Mol. Biol.}, 48:\penalty0 443--453, 1970.
\newblock \doi{10.1016/0022-2836(70)90057-4}.

\bibitem[Notredame et~al.(2000)Notredame, Higgins, and Heringa]{Notredame:00}
C{\'e}dric Notredame, Desmond~G Higgins, and Jaap Heringa.
\newblock {T-coffee}: a novel method for fast and accurate multiple sequence
  alignment.
\newblock \emph{Journal of molecular biology}, 302:\penalty0 205--217, 2000.
\newblock \doi{10.1006/jmbi.2000.4042}.

\bibitem[Otto et~al.(2011)Otto, Stadler, and Prohaska]{Otto:11a}
Wolfgang Otto, Peter~F. Stadler, and Sonja~J. Prohaska.
\newblock Phylogenetic footprinting and consistent sets of local aligments.
\newblock In R.~Giancarlo and G.~Manzini, editors, \emph{CPM 2011}, volume 6661
  of \emph{Lecture Notes in Computer Science}, pages 118--131, Heidelberg,
  Germany, 2011. Springer-Verlag.
\newblock \doi{10.1007/978-3-642-21458-5_12}.

\bibitem[Rautiainen and Marschall(2017)]{Rautiainen:17}
Mikko Rautiainen and Tobias Marschall.
\newblock {Aligning sequences to general graphs in $O(V + mE)$ time}.
\newblock Technical report, bioRxiv, 2017.

\bibitem[Retzlaff and Stadler(2018)]{Retzlaff:18a}
Nancy Retzlaff and Peter~F. Stadler.
\newblock Partially local multi-way alignments.
\newblock \emph{Math. Comp. Sci.}, 12:\penalty0 207--234, 2018.
\newblock \doi{10.1007/s11786-018-0338-4}.

\bibitem[Sankoff and Kruskal(1983)]{Sankoff:83}
David Sankoff and Joseph Kruskal, editors.
\newblock \emph{Time Warps, String Edits and Macromolecules: the theory and
  practice of Sequence Comparison}.
\newblock Addison-Wesley, London, U.K., 1983.

\bibitem[Smith and Waterman(1981)]{Smith:81}
Temple~F Smith and Michael~S Waterman.
\newblock Comparison of biosequences.
\newblock \emph{Adv. Appl. Math.}, 2:\penalty0 482--489, 1981.
\newblock \doi{10.1016/0196-8858(81)90046-4}.

\bibitem[Steiner et~al.(2011)Steiner, Stadler, and Cysouw]{Steiner:11a}
Lydia Steiner, Peter~F Stadler, and Michael Cysouw.
\newblock A pipeline for computational historical linguistics.
\newblock \emph{Language Dynamics \& Change}, 1:\penalty0 89--127, 2011.
\newblock \doi{10.1163/221058211X570358}.

\bibitem[Stoye et~al.(1997)Stoye, Moulton, and Dress]{Stoye:97}
Jens Stoye, Vincent Moulton, and Andreas W~M Dress.
\newblock {DCA}: an efficient implementation of the divide-and-conquer approach
  to simultaneous multiple sequence alignment.
\newblock \emph{Comput. Appl. Biosci.}, 13:\penalty0 625--626, 1997.
\newblock \doi{10.1093/bioinformatics/13.6.625}.

\bibitem[Tiepmar and Heyer(2017)]{Tiepmar:17}
Jochen Tiepmar and Gerhard Heyer.
\newblock An overview of canonical text services.
\newblock \emph{Linguistics Literature Studies}, 5:\penalty0 132--148, 2017.
\newblock \doi{10.13189/lls.2017.050209}.

\bibitem[Vaddadi et~al.(2017)Vaddadi, Sivadasan, Tayal, and
  Srinivasan]{Vaddadi:17}
Kavya Vaddadi, Naveen Sivadasan, Kshitij Tayal, and Rajgopal Srinivasan.
\newblock Sequence alignment on directed graphs.
\newblock Technical report, bioRxiv, 2017.

\bibitem[Velandia-Huerto et~al.(2016)Velandia-Huerto, Berkemer, Hoffmann,
  Retzlaff, Romero~Marroqu{\'\i}n, Hern{\'a}ndez~Rosales, Stadler, and
  Berm{\'u}dez-Santana]{VelandiaHuerto:16b}
Cristian~A Velandia-Huerto, Sarah~J Berkemer, Anne Hoffmann, Nancy Retzlaff,
  Liliana~C Romero~Marroqu{\'\i}n, Maribel Hern{\'a}ndez~Rosales, Peter~F
  Stadler, and Clara~I Berm{\'u}dez-Santana.
\newblock Orthologs, turn-over, and remolding of {tRNAs} in primates and fruit
  flies.
\newblock \emph{BMC Genomics}, 17:\penalty0 617, 2016.
\newblock \doi{10.1186/s12864-016-2927-4}.

\bibitem[Wang and Jiang(1994)]{Wang:94}
L~Wang and T~Jiang.
\newblock On the complexity of multiple sequence alignment.
\newblock \emph{J Comput Biol}, 1:\penalty0 337--348, 1994.
\newblock \doi{10.1089/cmb.1994.1.337}.

\bibitem[Wareham(1995)]{Wareham:95}
H~T Wareham.
\newblock A simplified proof of the {NP}- and {MAX SNP}-hardness of multiple
  sequence tree alignment.
\newblock \emph{J Comput Biol.}, 2:\penalty0 509--514, 1995.
\newblock \doi{10.1089/cmb.1995.2.509}.

\bibitem[Wolff(2000)]{Wolff:00}
J~G Wolff.
\newblock Syntax, parsing and production of natural language in a framework of
  information compression by multiple alignment, unification and search.
\newblock \emph{J. Universal Comp. Sci.}, 6\penalty0 (8):\penalty0 781--829,
  2000.
\newblock \doi{10.3217/jucs-006-08-0781}.

\end{thebibliography}

\end{document}